\documentclass{amsart}
\usepackage{amsmath,amsfonts,amssymb,amsthm}
\usepackage{url}
\usepackage{graphicx,color}
\usepackage{enumerate}

\newtheorem{theorem}{Theorem}[section]
\newtheorem{lemma}[theorem]{Lemma}

\newtheorem{prop}[theorem]{Proposition}

\newtheorem*{theorem*}{Theorem}

\theoremstyle{definition}
\newtheorem{definition}{Definition}[section]

\theoremstyle{remark}

\numberwithin{equation}{section}

\author{}
\address{}


\begin{document}

\title[$8$-located simplicial complexes with the SD'-property]{Classifying spaces for families of subgroups of $8$-located groups}

\author[Ioana-Claudia Laz\u{a}r]{
Ioana-Claudia Laz\u{a}r\\
Politehnica University of Timi\c{s}oara, Dept. of Mathematics,\\
Victoriei Square $2$, $300006$-Timi\c{s}oara, Romania\\
E-mail address: ioana.lazar@upt.ro}

\begin{abstract}

We investigate the structure of the minimal displacement set in $8$-located complexes with the SD'-property. We show that such set embeds isometrically into the complex. Since $8$-location and simple connectivity imply Gromov hyperbolicity, the minimal displacement set in such complex is systolic. Using these results, we construct a low-dimensional classifying space for the family of virtually cyclic subgroups of a group acting properly on an $8$-located complex with the SD'-property.

\end{abstract}

\maketitle

\section{Introduction}

Curvature can be expressed both in metric and combinatorial terms. Metrically, one can refer to ’nonpositively
curved’ (respectively, ’negatively curved’) metric spaces in the sense of Aleksandrov, i.e. by comparing small triangles in the space with
triangles in the Euclidean plane (hyperbolic plane). These are the CAT(0) (respectively, CAT(-1)) spaces. Combinatorially, one looks for local combinatorial conditions implying
some global features typical for nonpositively curved metric spaces.

A very important combinatorial condition of this type was formulated by Gromov \cite{Gro} for cubical complexes, i.e. cellular complexes
with cells being cubes. Namely, simply connected cubical complexes with links (that can be thought as small spheres around vertices)
being flag (respectively, $5$-large, i.e.\ flag-no-square) simplicial complexes carry a canonical CAT(0) (respectively, CAT(-1)) metric.
Another important local combinatorial condition is local $k$-largeness, introduced by Januszkiewicz-{\' S}wi{\c a}tkowski \cite{JS1} and Haglund \cite{Hag}. A flag simplicial complex is \emph{locally $k$-large} if its links do not contain `essential' loops of length less than $k$.
In particular, simply connected locally $7$-large simplicial complexes, i.e.\ \emph{$7$-systolic} complexes, are Gromov hyperbolic \cite{JS1}.
The theory of \emph{$7$-systolic groups}, that is groups acting geometrically on $7$-systolic complexes, allowed to provide important examples of highly dimensional Gromov hyperbolic groups \cite{JS0,JS1,O-chcg,OS,Prz,E1}.

However, for groups acting geometrically on CAT(-1) cubical complexes or on $7$-systolic complexes, some very restrictive limitations
are known. For example, $7$-systolic groups are in a sense `asymptotically hereditarily aspherical', i.e.\ asymptotically they can not contain
essential spheres. This yields in particular that such groups are not fundamental groups of negatively curved manifolds of dimension
above two; see e.g. \cite{JS2,O-ci,O-ib,OS,Gom,O-ns}.
This rises need for other combinatorial conditions, not imposing restrictions as above. In \cite{O-sdn,ChOs,BCCGO,ChaCHO} some conditions
of this type are studied -- they form a way of unifying CAT(0) cubical and systolic theories.

On the other hand, Osajda \cite{O-8loc} introduced a local combinatorial condition of \emph{$8$-location}, and used it to provide a new solution to Thurston's
problem about hyperbolicity of some $3$-manifolds. In \cite{L-8loc} we study of a version of $8$-location, suggested in \cite[Subsection 5.1]{O-8loc}.
This $8$-location says that homotopically trivial loops of length at most $8$ admit filling diagrams with one internal vertex.
However, in the new $8$-location essential $4$-loops are allowed.
In \cite{L-8loc} (Theorem $4.3$) it is shown that simply connected,
$8$-located simplicial complexes are Gromov hyperbolic. In the current paper we give an application to this result.

We focus on the study of the minimal displacement set in an $8$-located complex satisfying the $SD'$-property.
One of the paper's results states that such set is isometrically embedded into the complex. Moreover, we show that such set is Gromov hyperbolic. In particular, it is systolic. This follows as an application of the fact that $8$-located complexes with the $SD'$-property are Gromov hyperbolic (see \cite{L-8loc}).

For CAT(0) spaces and systolic complexes, however, studying the structure of the minimal displacement set is useful when constructing a low-dimensional classifying space for the family of virtually cyclic subgroups of a group acting properly on a CAT(0) space, respectively on a systolic complex (see \cite{BH}, \cite{OP}). We expect similar results in the $8$-located case.  Knowing that the minimal displacement set of an $8$-located complex with the $SD'$-property embeds isometrically into the complex and it is systolic, we will be able to apply results proven in\cite{E2} and \cite{OP} on systolic complexes.

\textbf{Acknowledgements}.
The author would like to thank Victor Chepoi, Damian Osajda and Tomasz Prytu{\l}a for useful discussions.
This work was partially supported by the grant $346300$ for IMPAN from the Simons Foundation and the matching $2015-2019$ Polish MNiSW fund.

\section{Preliminaries}

Let $X$ be a simplicial complex. We denote by $X^{(k)}$ the $k$-skeleton of
$X, 0 \leq k < \dim X$. A subcomplex $L$ in $X$ is called \emph{full} as a subcomplex of $X$ if any simplex of $X$ spanned by a set of vertices in $L$, is a simplex of $L$. For a set
$A = \{ v_{1}, ..., v_{k} \}$ of vertices of $X$, by $\langle  A \rangle$ or by $\langle  v_{1}, ..., v_{k} \rangle$ we denote the \emph{span} of $A$, i.e. the
smallest full subcomplex of $X$ that
contains $A$. We write $v \sim v'$ if $\langle  v,v' \rangle \in X$ (it can happen that $v = v'$). We write $v \nsim v'$ if $\langle  v,v' \rangle \notin X$.
 We call $X$ {\it flag} if any finite set of vertices which are pairwise connected by
edges of $X$, spans a simplex of $X$.

A {\it cycle} ({\it loop}) $\gamma$ in $X$ is a subcomplex of $X$ isomorphic to a triangulation of $S^{1}$. A \emph{full cycle} in $X$ is a cycle that is full as a subcomplex of $X$.
A $k$-\emph{wheel} in $X$ $(v_{0}; v_{1}, ..., v_{k})$ (where $v_{i}, i \in \{0,..., k\}$
are vertices of $X$) is a subcomplex of $X$ such that $\gamma = (v_{1}, ..., v_{k})$ is a full cycle and $v_{0} \sim v_{1}, ..., v_{k}$.
The \emph{length} of $\gamma$ (denoted by $|\gamma|$) is the number of edges in $\gamma$. Given two cycles $\alpha, \beta$ of $X$, we denote by $\alpha \star \beta$ their concatenation.

We define the \emph{combinatorial metric} on the $0$-skeleton of $X$ as the number of edges in the shortest $1$-skeleton path joining two given vertices.

A \emph{ball (sphere)}
$B_{i}(v,X)$ ($S_{i}(v,X)$) of radius $i$ around some vertex $v$ is a full subcomplex of $X$ spanned by vertices at combinatorial distance at most $i$ (at combinatorial distance $i$) from $v$.

\begin{definition}\label{def-2.2}
A simplicial complex is $m$-\emph{located}, $m \geq 4$, if it is flag and every full homotopically trivial loop of length at most $m$ is contained in a $1$-ball.
\end{definition}

Let $\sigma$ be a simplex of $X$. The \emph{link} of $X$ at $\sigma$, denoted by $X_{\sigma}$, is the subcomplex of $X$ consisting of all simplices of $X$ which are disjoint from $\sigma$ and which, together
with $\sigma$, span a simplex of $X$. We call a flag simplicial complex $k$\emph{-large} if there are no
full $j$-cycles in $X$, for $j<k$. We say $X$ is \emph{locally} $k$\emph{-large} if all its links are $k$-large. We call $X$ $k$\emph{-systolic} if it is connected, simply connected and \emph{locally} $k$\emph{-large}. For $k = 6$, we abbreviate $k$-systolic to systolic.

We introduce further a global combinatorial condition on a flag simplicial complex.

\begin{definition}\label{def-2.3}
Let $X$ be a flag simplicial complex. For a vertex $O$ of $X$ and a natural number $n$, we say that $X$ satisfies \emph{the property $SD'_{n}(O)$}
if for every $i \in \{1, ..., n\}$ we have:
\begin{enumerate}
\item (T) (triangle condition): for every edge $e \in S_{i+1}(O)$, the intersection $X_{e} \cap B_{i}(O)$ is non-empty;
\item (V) (vertex condition): for every vertex $v \in S_{i+1}(O)$, and for every two vertices $u,w \in X_{v} \cap B_{i}(O)$,
there exists a vertex $t \in X_{v} \cap B_{i}(O)$ such that $t \sim u,w$. \end{enumerate}
We say $X$ satisfies \emph{the property $SD'(O)$} if $SD'_{n}(O)$ holds for each natural number $n$. We say $X$ satisfies \emph{the property $SD'$} if
$SD'_{n}(O)$ holds for each natural number $n$ and for each vertex $O$ of $X$.
\end{definition}

The following result is given in \cite{O-8loc}.

\begin{prop}\label{2.1a}
A simplicial complex which satisfies the property $SD'(O)$ for some vertex $O$, is simply connected.
\end{prop}

\begin{definition}
A group acting properly discontinously and cocompactly, by automorphisms,
on an $m$-located simplicial complex with the $SD'$ property, is called an \emph{$m$-located} group, $m \geq 4$.
\end{definition}

\begin{definition}\label{def-2.4}
Given a path $\gamma = (v_{0}, v_{1}, ..., v_{n})$ in a simplicial complex $X$, one can \emph{tighten} it to a full path $\gamma'$ with the same endpoints by repeatedly applying the following operations:

$\bullet$ if $v_{i}$ and $v_{j}$ are adjacent in $X$ for some $j > i+1$, then remove from the sequence all $v_{k}$ where $i < k < j$;

$\bullet$ if $v_{i}$ and $v_{j}$ coincide in $X$ for some $j > i$, then remove from the sequence all $v_{k}$ where $i < k \leq j$.

We call $\gamma'$ a \emph{tightening} of $\gamma$. We allow the trivial case when $\gamma$ is already full. Then its tightening is the path itself.
\end{definition}

\subsection{Minimal displacement set for CAT(0) spaces}

For CAT(0) spaces the minimal displacement set is studied in \cite{BH}.

\begin{definition} Let $X$ be a metric space and let $h$ be an isometry of $X$. The \emph{displacement function} of $h$ is the function $d_{h} : X \rightarrow \mathbf{R}$ defined by $d_{h}(x) = d(h(x),x)$.
The \emph{translation length} of $h$ is the number $|h| = \inf \{ d_{h}(x) | x \in X\}$. The set of
points where $d_{h}$ attains this infimum is denoted by $\rm{Min}_{X}(h)$ and it is called the \emph{minimal displacement set}.
\end{definition}

\begin{definition} Let $X$ be a metric space. An isometry $h$ of 
 X is called
\begin{enumerate}
\item \emph{elliptic} if $h$ has a fixed point,
\item \emph{hyperbolic} if $d_{h}$ attains a strictly positive minimum.
\end{enumerate}
\end{definition}

The following result concerns the structure of the minimal displacement set of a hyperbolic isometry $h$ in a CAT(0) space.

\begin{theorem}
Let $X$ be a CAT(0) space.
\begin{enumerate}
\item An isometry $h$ of $X$ is hyperbolic if and only if there exists a geodesic line
$c : \mathbf{R} \rightarrow X$ which is translated non-trivially by $h$; namely $h(c(t)) = c(t+a)$, for
some $a > 0$. The set $c(\mathbf{R})$ is called an axis of $h$. For any such axis, the number
$a$ is equal to $|h|$.

Let $h$ be a hyperbolic isometry of $X$.
\item The axes of $h$ are parallel to each other and their union is $\rm{Min}_{X}(h)$.
\item $\rm{Min}_{X}(h)$ is isometric to a product $Y \times \mathbf{R}$, and the restriction of $h$ to $\rm{Min}_{X}(h)$ is
of the form $(y, t) \rightarrow (y, t+|h|)$, where $y \in Y$, $t \in \mathbf{R}$ (see \cite{BH}, Theorem $6.8$, page $231$).
\end{enumerate}
\end{theorem}

\subsection{Minimal displacement set for systolic complexes}

For systolic complexes the minimal displacement set is studied in \cite{E2}.

Let $h$ be an isometry of a simplicial complex $X$. We define the \emph{displacement function}
$d_{h} : X^{(0)} \rightarrow \mathbf{N}$ by $d_{h}(x) = d_{X}(h(x),x)$. The \emph{translation length} of $h$ is
defined as $|h| = \min_{x \in X^{(0)}} d_{h}(x)$. 
If $h$ does not fix any simplex of $X$, then $h$ is called \emph{hyperbolic}. In such case one has
$|h| > 0$. Otherwise we call the isometry $h$ \emph{elliptic}. For a hyperbolic isometry $h$, we define the minimal displacement set $\rm{Min}_{X}(h)$
as the subcomplex of $X$ spanned by the set of vertices where $d_{h}$ attains its minimum.
Clearly $\rm{Min}_{X}(h)$ is invariant under the action of $h$. 

\begin{theorem}
Let $h$ be a hyperbolic isometry of a
systolic complex $X$. Then the subcomplex $\rm{Min}_{X}(h)$ is a systolic subcomplex, isometrically
embedded into $X$ (see \cite{E2}, Propositions $3.3$ and $3.4$). 
\end{theorem}

Let $h$ be an isometry of a simplicial complex $X$.
An $h$-invariant geodesic in $X$ is called an \emph{axis} of $h$. We say that $\rm{Min}_{X}(h)$ is
the union of axes, if for every vertex $x \in \rm{Min}_{X}(h)$, there is an $h$-invariant geodesic
passing through $x$, i.e. $\rm{Min}_{X}(h)$ can be written as follows:
\begin{center}
$\rm{Min}_{X}(h) = \rm{span}\{ \bigcup \gamma | \gamma$ is an $h$-invariant geodesic $\}$ ($2.1$)
\end{center}
In this case, the isometry $h$ acts on $X$ as a translation along the axes by the
number $|h|$.

For two subcomplexes $X_{1}, X_{2} \subset X$, the distance $d_{\min}(X_{1}, X_{2})$ is defined to be
\begin{center}
    $d_{\min}(X_{1}, X_{2}) = \min \{ d_{X}(x_{1},x_{2}) | x_{1} \in X_{1}, x_{2} \in X_{2} \}$.
\end{center}

Next we define
the \emph{graph of axes} denoted by $Y_{h}$.
For a hyperbolic isometry $h$ satisfying $(2.1)$,
we define the simplicial graph $Y_{h}$ as follows:
\begin{center}
    $Y_{h}^{(0)} = \{ \gamma | \gamma$ is an $h$-invariant geodesic in $\rm{Min}_{X}(h) \}$, \\
    
     $Y_{h}^{(1)} = \{ \{\gamma_{1},\gamma_{2}\} | d_{\min}(\gamma_{1}, \gamma_{2}) \leq 1 \}$.
\end{center}
Let $d_{Y(h)}$ denote the associated metric on $Y_{h}^{(0)}$.

\subsection{Hyperbolicity}

One of the paper's main results relies on the following theorem.

\begin{theorem}\label{2.4}
Let $X$ be an $8$-located simplicial complex which satisfies the $SD'$ property. Then the $0$-skeleton of $X$ with a path metric induced from $X^{(1)}$, is $\delta$-hyperbolic, for a
universal constant $\delta$  (see \cite{L-8loc}, Theorem $3.7$).
\end{theorem}

We shall apply the following lemmas frequently.

\begin{lemma}\label{2.5}
Let $X$ be an $8$-located simplicial complex which satisfies the $SD'_{n}(O)$ property for some vertex $O$, $n \geq 2$. Let $v \in S_{n+1}(O)$ and let $y,z \in B_{n}(O)$ be such that $v \sim y,z$
and $d(y,z) = 2$.
Let $w \in B_{n}(O)$ be a vertex such that $w \sim y, v, z$, given by the vertex condition (V). We consider the vertices $u_{1}, u_{2} \in B_{n-1}(O)$ such that
$u_{1} \sim  y, w$ and $u_{2} \sim  w, z$, given by the triangle condition (T). If $u_{1} \nsim z$ and $u_{2} \nsim y$, then
$u_{1} \sim u_{2}$ (see \cite{L-8loc}, Lemma $3.1$). 
\end{lemma}

\begin{lemma}\label{2.6}
Let $X$ be an $8$-located simplicial complex which satisfies the $SD'_{n}(O)$ property for some vertex $O$, $n \geq 2$. Let $v_{1}, v_{2}, v_{3} \in B_{n-1}(O)$ be such that $v_{1} \sim v_{2} \sim v_{3}$. Let $w_{1},w_{2} \in B_{n-2}(O)$ be such that $w_{1} \sim v_{1},v_{2}$ and
$w_{2} \sim v_{2},v_{3}$, given by the triangle condition (T).
Let $p_{1},p_{2} \in B_{n}(O)$ be such that $p_{1} \sim v_{1},v_{2}$ and
$p_{2} \sim v_{2},v_{3}$, given by the triangle condition (T).
Then $w_{1} \sim w_{2}$ if and only if
$p_{1} \sim p_{2}$  (see \cite{L-8loc}, Lemma $3.2$).
\end{lemma}

\subsection{Classifying spaces with finite or virtually cyclic stabilisers.} 
The main
goal of this section is, given a group $G$, to describe a method of constructing a
model for a classifying space
with virtually cyclic stabilisers out of a model for
a classifying space
with finite stabilisers. The presented method is
due to W. L\"uck and M. Weiermann (\cite{LW}). First we give the necessary definitions. 

A collection of subgroups $\mathcal{F}$ of a group $G$ is called a \emph{family} if it is closed under
taking subgroups and conjugation by elements of $G$. Two examples which will be of interest to us are the family $\mathcal{FIN}$ of all finite subgroups, and the family $\mathcal{VCY}$ of all
virtually cyclic subgroups.

\begin{definition}
Given a group $G$ and a family of its subgroups $\mathcal{F}$, a \emph{model for the classifying space} $E_{\mathcal{F}}G$ is a $G$-$CW$-complex $X$ such that for any subgroup $H \subset G$
the fixed point set $X^{H}$ is contractible if $H \in \mathcal{F}$, and empty otherwise.
\end{definition}

Let $\underline{E}G$ denote $E_{\mathcal{FIN}}G$
and let $\underline{\underline{E}}G$ denote $E_{\mathcal{VCY}}G$. 

A model for $E_{\mathcal{F}}G$ exists for any group and any family. Any two
models for $E_{\mathcal{F}}G$ are $G$-homotopy equivalent (see \cite{L}).
However, general constructions always produce infinite dimensional models.

We will describe a method of constructing a finite dimensional model for $\underline{\underline{E}}G$
out of a model for $\underline{E}G$ and appropriate models associated to infinite virtually cyclic subgroups of $G$. If $H \subset G$ is a subgroup and $\mathcal{F}$ is a family of subgroups of $G$, let $\mathcal{F} \cap H$
denote the family of all subgroups of $H$ which belong to the family $\mathcal{F}$. More
generally, if $\phi : H \rightarrow G$ is a homomorphism, let
$\phi^{\star}\mathcal{F}$ denote the smallest family
of subgroups of $H$ that contains 
$\phi^{-1}(F)$ for all $F \in \mathcal{F}$.

Consider the collection $\mathcal{VCY} \setminus \mathcal{FIN}$ of infinite virtually cyclic subgroups of $G$. It
is not a family since it does not contain the trivial subgroup. Define an equivalence
relation on $\mathcal{VCY} \setminus \mathcal{FIN}$ by 
\begin{center}
$H_{1} \sim H_{2}$ $\iff |H_{1} \cap H_{2}| = \infty$
\end{center}
Let $[H]$ denote the equivalence class of $H$, and let $[\mathcal{VCY} \subset \mathcal{FIN}]$ denote the set of
equivalence classes. The group $G$ acts on $[\mathcal{VCY} \subset \mathcal{FIN}]$ by conjugation, and for a
class $[H] \in  [\mathcal{VCY} \subset \mathcal{FIN}]$ define the subgroup $N_{G}(H) \subseteq G$ to be the stabiliser of
$[H]$ under this action, i.e. \begin{center} $N_{G}(H) = \{ g \in G | $ $ |g^{-1}Hg \cap H| = \infty \}$ \end{center}
The subgroup $N_{G}(H)$ is called the \emph{commensurator} of $H$, since its elements
conjugate $H$ to the subgroup commensurable with $H$. For $[H] \in [\mathcal{VCY} \subset \mathcal{FIN}]$
define the family $\mathcal{G}[H]$ of subgroups of $N_{G}(H)$ as follows
\begin{center}
    $\mathcal{G}[H] = \{ K \subset G | $ $ K \in [\mathcal{VCY} \subset \mathcal{FIN}], [K] = [H] \} \cup \{ K \in \mathcal{FIN} \cap N_{G}[H] \}$.
\end{center}

\begin{definition}
A group G satisfies \emph{condition (C)} if for every
$g,h \in G$ with $|h| = \infty$ (infinite order) and any $k,l \in \mathbf{Z}$ we have \begin{center} $g h^{k} g^{-1} = h^{l} \implies |k| = |l|$ \end{center}
\end{definition}

\begin{lemma}\label{2.1}
Let $K \subset N_{G}[H]$ be a finitely generated subgroup that contains some
representative of $[H]$ and assume that the group $G$ satisfies condition (C). Choose
an element $h \in H$ such that $[\langle h \rangle] = [H]$ (any element of infinite order has this
property). Then there exists $k \geq 1$, such that $\langle h^{k} \rangle$ is normal in $K$.
\end{lemma}

For the proof see \cite{OP}, Lemma $2.6$, page $8$.

\section{Minimal displacement set for $8$-located complexes with the SD'-property}

We study the structure of the minimal displacement set in an $8$-located complex with the SD'-property. The notations introduced in section $2.2$ hold in this section as well.

\begin{lemma}\label{3.1}
Let $h$ be a simplicial isometry without fixed points of a simplicial complex $X$. We choose
a vertex $v \in \rm{Min}_{X}(h)$ and a geodesic $\alpha \subset X^{(1)}$ joining $v$ with $h(v)$. Consider a simplicial path
$\gamma : \mathbf{R} \rightarrow X$ (where $\mathbf{R}$ is given a simplicial structure with $\mathbf{Z}$ as the set of vertices) being the
concatenation of geodesics $h^{n}(\alpha), n \in \mathbf{Z}$.
Then $\gamma$ is a $|h|$-geodesic (i.e. $d(\gamma(a), \gamma(b)) = |a-b|$
if $a, b$ are such integers that $|a - b| \leq |h|$). In particular, $\rm{Im}(\gamma) \subset \rm{Min}_{X}(h)$.
\end{lemma}

\begin{proof}
 The proof is similar to the one given in \cite{E2}, Fact $3.2$.
We prove the statement for $|a-b| = |h|$ (this implies the general case). Then, by the
construction of $\gamma$, either $\gamma(b) = h(\gamma(a))$ or $\gamma(a) = h(\gamma(b))$. Thus we have $d(\gamma(a), \gamma(b)) \geq |h|$.
The opposite inequality follows from the fact that $\gamma$ is a simplicial map.
\end{proof}

Next we prove one of the paper's main results.

\begin{theorem}\label{3.2}
Let $h$ be a (simplicial) isometry with no fixed points of an $8$-located complex $X$ with the SD'-property. Assume $|h| > 3$. Then
the $1$-skeleton of $\rm{Min}_{X}(h)$ is isometrically embedded into $X$.
\end{theorem}

\begin{proof}

The construction is similar to the one given in \cite{E2}, Proposition $3.3$ for systolic complexes.

Suppose the $1$-skeleton of $\rm{Min}_{X}(h)$ is not isometrically embedded. Then there exist
vertices $v, w \in \rm{Min}_{X}(h)$ such that no geodesic in $X$ with endpoints $v$ and $w$ is contained in
$\rm{Min}_{X}(h)$. Choose $v$ and $w$ so that $d(v,w)$ minimal (clearly $d(v,w) > 1$). Join $v$ with
$h(v)$, $w$ with $h(w)$ and $v$ with $w$ by geodesics $\alpha, \beta$ and $\gamma$, respectively. Then $h(v)$ is joined
with $h(w)$ by $h(\gamma)$. Note that $l(\alpha) = l(\beta) = |h|$, $l(\gamma) = l(h(\gamma)) > 1$.

According to Lemma \ref{3.1}, we have
$\alpha, \beta \subset \rm{Min}_{X}(h)$.
Then, by minimality of $d(v, w)$, geodesics $\alpha$ and $\gamma$ intersect only at the endpoints. The same holds for the geodesics $\alpha$ and $h(\gamma)$, $\beta$ and $\gamma$, $\beta$ and $h(\gamma)$, respectively.
Suppose there is a vertex $x \in \gamma \cap h(\gamma)$. Then $h(x) \in h(\gamma)$ and $h(x) \neq x$, since $h$ has no
fixed points. We may assume, not losing generality, that $h(v), x, h(x)$ and $h(w)$ lie on $h(\gamma)$
in this order. Then
$d(x,h(x)) = d(h(v),h(x)) - d(h(v),x) = d(v,x) - d(h(v),x) \leq d(v,h(v)) = |h|$.
So $x \in \rm{Min}_{X}(h)$, contradicting the minimality of $d(v,w)$.
Thus the geodesics $\alpha, \beta, \gamma, h(\gamma)$ either are pairwise disjoint but the endpoints or $\alpha$ and
$\beta$ have nonempty intersection. In both situations we proceed as follows.

Let $y,x$ be adjacent vertices on $\gamma$ such that $d(y,v) = d(x,v) - 1$. It may happen that $y = v$ or $x = w$ but not simultaneously due to the fact that $d(v,w) > 1$. The vertex $y$ is the last vertex of $\gamma$ such that $d(y,h(y)) = d(y,v) + d(v,h(v)) + d(h(v),h(y))$ (i.e.
$y$ is the last vertex of $\gamma$ to be joined with $h(y)$ by the left of the cycle $\gamma \star \beta \star h(\gamma) \star \alpha$). The vertex $x$ is the first vertex of $\gamma$ such that $d(x,h(x)) = d(x,w) + d(w,h(w)) + d(h(w),h(x))$ (i.e. $x$ is the first vertex of $\gamma$ to be joined with $h(x)$ by the right of the cycle $\gamma \star \beta \star h(\gamma) \star \alpha$). 
 Let $v' \in \gamma, v' \sim v$ (possibly with $v' = y$).

There are two cases: either $l(\gamma) = 2$ or $l(\gamma) \geq 3$.

Assume first $l(\gamma) = 2$. Then $y=x$. Note that $d(v,h(y)) = d(w,h(y)) = |h| + 1$. 

Note that $y \in B_{2+|h|}(h(y))$. Because $v,w \in X_{y} \cap B_{1+|h|}(h(y))$, the (V) condition of the SD'(h(y))-property implies that there exists a vertex $t \in X_{y} \cap B_{1+|h|}(h(y))$ such that $t \sim v,w$.

Because $v,t \in B_{1+|h|}(h(y))$, $v \sim t$, the (E) condition of the SD'(h(y))-property implies that there exists a vertex $p \in B_{|h|}(h(y))$ such that $p \sim v,t$.

Because $t,w \in B_{1+|h|}(h(y))$, $t \sim w$, the (E) condition of the SD'(h(y))-property implies that there exists a vertex $q \in B_{|h|}(h(y))$ such that $q \sim t,w$.

Note that $y \in B_{2 +|h|} (h(y))$, $v,t,w \in X_{y} \cap B_{1 +|h|} (h(y))$,  $p,q \in B_{|h|} (h(y)), p \sim v,t; q \sim t,w$. Then Lemma \ref{2.5} implies that $p \sim q$.

Let $l \in \beta$ such that $w \sim l$.  Because $q,l \in X_{w} \cap B_{|h|}(h(y))$, the (V) condition of the SD'(h(y))-property implies that there exists a vertex $r \in X_{w} \cap B_{|h|}(h(y))$ such that $r \sim q,l$.

Because $p,q \in B_{|h|}(h(y))$, $p \sim q$, the (E) condition of the SD'(h(y))-property implies that there exists a vertex $m \in B_{|h|-1}(h(y))$ such that $m \sim p,q$.

Because $q,r \in B_{|h|}(h(y))$, $q \sim r$, the (E) condition of the SD'(h(y))-property implies that there exists a vertex $n \in B_{|h|-1}(h(y))$ such that $n \sim q,r$.

Note that $t,w \in B_{1+|h|}(h(y))$, $p,q,r \in B_{|h|}(h(y))$, $m,n \in X_{q} \cap  B_{|h|-1}(h(y)), p,q \in X_{t}$, $q, r \in X_{w}$. Then, because $t \sim w$, Lemma \ref{2.6} implies that $m \sim n$.

Let $\delta$ be the tightening of the cycle $(y,v,p,m,n,r,w)$. Note that $|\delta| \leq 7$ and the cycle $\delta$ is full. Then, by $8$-location, there is a vertex $f$ such that $\delta \subset X_{f}$. Hence $d(y,m) = 2$. But $y \in B_{2 + |h|}(h(y))$ while $m \in B_{|h|-1}(h(y))$. Therefore $d(y,m) = 3$. This yields a contradiction.

For the rest of the proof let $l(\gamma) \geq 3$.

Note that either $d(v',h(v')) = |h|+2$ or $d(v',h(v')) = |h|+1$ or $d(v',h(v')) = |h|$. We analyze these cases below.

Case A. Suppose $d(v',h(v')) = |h|+2$. So there do not exist vertices $a,b \in \alpha$ such that $v' \sim a \sim v$, $h(v') \sim b \sim h(v)$.

Case A.1. Assume $|\gamma| = 2k, k \in \mathbf{N}^{\star}$. 

Assume w.l.o.g. $d(y,v) = k$. Then, due to the choice of the vertices $x$ and $y$, we have $d(x,w) = k-1$. Recall $y$ is the last vertex of $\gamma$ to be joined with $h(y)$ by the left of the cycle $\gamma \star \beta \star h(\gamma) \star \alpha$; $x$ is the first vertex of $\gamma$ to be joined with $h(x)$ by the right of the cycle $\gamma \star \beta \star h(\gamma) \star \alpha$.

Let $z \in \gamma$ such that $z \sim y$, $d(z,v) = d(y,v) - 1$. Note that $d(z,h(y)) = d(x,h(y)) = 2k- 1 + |h|$. Hence $z,x \in X_{y} \cap B_{ 2k- 1 + |h|}(h(y))$. Then the (V) condition of the SD'($h(y)$)-property implies that there exists a vertex $t \sim x,z$ such that $t \in X_{y} \cap B_{2k- 1 + |h|}(h(y))$.

Note that $z,t \in B_{2k- 1 + |h|}(h(y))$ and $z \sim t$. Then, by the (E) condition of the SD'($h(y)$)-property, there exists $p \in B_{ 2k - 2 + |h|}(h(y))$ such that $p \sim z,t$.

Note that $t, x \in B_{ 2k- 2 + |h|}(h(y))$ and $t \sim x$. Then, by the (E) condition of the SD'($h(y)$)-property, there exists $q \in B_{ 2k- 1 + |h|}(h(y))$ such that $q \sim t,x$.

Note that $y \in B_{2k +|h|} (h(y))$, $z,t,x \in X_{y} \cap B_{2k-1 +|h|} (h(y))$,  $p,q \in X_{t} \cap B_{2k-2+|h|} (h(y)$, $p \sim z, q \sim x$. Then Lemma \ref{2.5} implies that $p \sim q$.

If $|\gamma| = 3$, let $u = w$.
If $|\gamma| > 3$, let $u \in \gamma$ such that $x \sim u$, $d(u,w) = d(x,w) - 1$. Note that $d(q,h(y)) = d(u,h(y)) = 2k - 2 + |h|$. Hence $q,u \in X_{x} \cap B_{2k - 2 + |h|}(h(y))$. Then the (V) condition of the SD'($h(y)$)-property implies that there exists a vertex $r \sim q,u$ such that $r \in X_{x} \cap B_{2k - 2 + |h|}(h(y))$. 

Note that $p,q \in B_{2k- 2 + |h|}(h(y))$ and $p \sim q$. Then by the (E) condition of the SD'($h(y)$)-property, there exists $m \in B_{2k - 3 + |h|}(h(y))$ such that $m \sim p,q$.

Note that $q,r \in B_{2k- 2 + |h|}(h(y))$ and $q \sim r$. Then by the (E) condition of the SD'($h(y)$)-property, there exists $n \in B_{2k - 3 + |h|}(h(y))$ such that $n \sim q,r$.

Note that $t,x \in B_{2k-1+|h|}(h(y))$, $p,q,r \in B_{2k-2+|h|}(h(y))$, $m,n \in $ $X_{q} \cap$ $ B_{2k-3+|h|}(h(y))$, $p,q \in X_{t}$, $q, r \in X_{x}$. Then, because $t \sim x$, Lemma \ref{2.6} implies that $m \sim n$.

Let $\delta$ be the tightening of the cycle $(y,z,p,m,n,r,x)$. Note that $|\delta| \leq 7$ and the cycle $\delta$ is full. Then, by $8$-location, there is a vertex $f$ such that $\delta \subset X_{f}$. Hence $d(y,m) = 2$. But $y \in B_{2k + |h|}(h(y))$ while $y \in B_{2k - 3 + |h|}(h(y))$. Therefore $d(y,m) = 3$. This yields a contradiction.

Case A.2. Assume $|\gamma| = 2k + 1, k \in \mathbf{N}^{\star}$. 

Assume first $d(v,y) = k+1$. Then $d(y,w) = k$. Note that $d(y,h(y)) = 2k + 1 + |h|$ by the left of the cycle $\gamma \star \beta \star h(\gamma) \star \alpha$ and $d(y,h(y)) = 2k + |h|$ by the right of the cycle $\gamma \star \beta \star h(\gamma) \star \alpha$. So the geodesic from $y$ to $h(y)$ passes by the right of the cycle $\gamma \star \beta \star h(\gamma) \star \alpha$. But the point $y$ is chosen such that the geodesic from $y$ to $h(y)$ passes by the left of the cycle $\gamma \star \beta \star h(\gamma) \star \alpha$. The situation $d(v,y) = k+1$ is therefore not possible. 
So the only possible case is when $d(v,y) = k$. Therefore $d(y,w) = k+1$, $d(x,w) = k$. 

Let $z \in \gamma$ such that $z \sim y$, $d(z,v) = d(y,v) - 1$. Note that $d(x,h(x)) = d(z,h(x)) = 2 k + |h|$. Because $x,z \in X_{y} \cap B_{2k+|h|} (h(x))$, the (V) condition of the SD'($h(x)$)-property implies that there exists a vertex $t \in X_{y} \cap B_{2k+|h|}(h(x))$ such that $t \sim x, z$.

Note that $z,t \in B_{2k + |h|}(h(x))$ and $z \sim t$. Then, by the (E) condition of the SD'($h(x)$)-property, there exists $p \in B_{2k - 1 + |h|}(h(y))$ such that $p \sim z,t$.

Note that $t,x \in B_{2k + |h|}(h(x))$ and $t \sim x$. Then, by the (E) condition of the SD'($h(x)$)-property, there exists $q \in B_{2k - 1 + |h|}(h(x))$ such that $q \sim t,x$.

Note that $y \in B_{2k + 1 +|h|} (h(x))$, $z,t,x \in X_{y} \cap B_{2k +|h|} (h(x))$,  $p,q \in X_{t} \cap B_{2k - 1 +|h|} (h(x)),$ $p \sim z, q \sim x$. Then Lemma \ref{2.5} implies that $p \sim q$.

If $|\gamma| > 5$,
let $l \in \gamma$ such that $z \sim l$, $d(l,v) = d(z,v) - 1$. If $|\gamma| \in \{3,5\}$, then $l \in \alpha$ such that $z \sim l$, $d(l,v) = d(z,v) - 1$. Note that $d(l,h(x)) = d(p,h(x)) = 2 k - 1 + |h|$. Because $l,p \in X_{z} \cap B_{2k - 1 +|h|} (h(x))$, the (V) condition of the SD'($h(x)$)-property implies that there exists a vertex $s \in X_{z} \cap B_{2k - 1 +|h|}(h(x))$ such that $s \sim l, p$.

Because $l,s \in B_{2k - 1 +|h|}(h(x)), l \sim s$, the (E) condition of the SD'($h(x)$)-property implies that there exists a vertex $m \in B_{2k-2+|h|}(h(x))$ such that $m \sim l,s$.

Because $s,p \in B_{2k - 1 +|h|}(h(x)), s \sim p$, the (E) condition of the SD'($h(x)$)-property implies that there exists a vertex $n \in B_{2k-2+|h|}(h(x))$ such that $n \sim s,p$.

Because $p,q \in B_{2k-1+|h|}(h(x)), p \sim q$, the (E) condition of the SD'($h(x)$)-property implies that there exists a vertex $r \in B_{2k-2+|h|}(h(x))$ such that $r \sim p,q$.

Note that $z \in B_{2k+|h|}(h(x))$, $l, s, p \in X_{z} \cap B_{2k-1+|h|}(h(x))$, $m, n \in X_{s} \cap B_{2k-2+|h|}(h(x))$, $m \sim l, n \sim p$. Then Lemma \ref{2.4} implies that $m \sim n$.

Note that $z, t \in B_{2k+|h|}(h(x))$, $s,p,q \in B_{2k-1+|h|}(h(x))$, $n, r \in X_{p} \cap B_{2k-2+|h|}(h(x))$, $s,p \in X_{z}, p,q \in X_{t}$, $n \sim s$, $r \sim q$. Then, because $z \sim t$, Lemma \ref{2.6} implies that $n \sim r$.

Let $\delta$ be the tightening of the cycle $(y,z,l,m,n,r,q,x)$. Note that $|\delta| \leq 8$ and the cycle $\delta$ is full. Then, by $8$-location, there is a vertex $f$ such that $\delta \subset X_{f}$. Hence $d(y,m) = 2$. But $y \in B_{2k + 1 + |h|}(h(y))$ while $m \in B_{2k - 2 + |h|}(h(y))$. Therefore $d(y,m) = 3$. This yields a contradiction.

In conclusion we have $d(v',h(v')) \neq |h|+2$. This completes case $A.$

Case $B.$ There exists a vertex $a \in \alpha, v \sim a \sim v'$. Suppose $d(v',h(v')) = |h|+1$.

Case B.1. Assume $|\gamma| = 2k, k \in \mathbf{N}^{\star}$. 

Assume w.l.o.g. $d(y,v) =k$. Then $d(y,h(y)) = 2 k - 1 + |h|$. Due to the choice of the vertices $x,y \in \gamma$, we have $d(x,w) = k-1$. Recall $y$ is the last vertex of $\gamma$ to be joined with $h(y)$ by the left of the cycle $\gamma \star \beta \star h(\gamma) \star \alpha$; $x$ is the first vertex of $\gamma$ to be joined with $h(x)$ by the right of the cycle $\gamma \star \beta \star h(\gamma) \star \alpha$. 

Note that $d(y,h(y)) = d(x,h(y)) = 2k-1+|h|$. Then $y,x \in B_{2k-1+|h|}(h(y))$. Because $y \sim x$, the (E) condition of the SD'($h(y)$)-property implies that there exists a vertex $t \in B_{2k-2+|h|}(h(y))$ such that $t \sim y,x$.

Let $l \in \gamma$ such that $x \sim l$, $d(l,w) = d(x,w) - 1$. Note that $t,l \in X_{x} \cap B_{2k-2+|h|}(h(y))$. Then the (V) condition of the SD'($h(y)$)-property implies that there exists a vertex $m \in X_{x} \cap B_{2k-2+|h|}(h(y))$ such that $m \sim t,l$.

Because $t,m \in B_{2k-2+|h|}(h(y)), t \sim m$, the (E) condition of the SD'($h(y)$)-property implies that there exists a vertex $r \in B_{2k-3+|h|}(h(y))$ such that $r \sim t,m$.

Because $m,l \in B_{2k-2+|h|}(h(y)), m \sim l$, the (E) condition of the SD'($h(y)$)-property implies that there exists a vertex $s \in B_{2k-3+|h|}(h(y))$ such that $s \sim m,l$.

Note that $x \in B_{2k-1+|h|}(h(y))$, $t,m,l \in X_{x} \cap B_{2k-2+|h|}(h(y))$ and $r,s \in X_{m} \cap B_{2k-3+|h|}(h(y))$, $r \sim t, s \sim l$. Then Lemma \ref{2.4} implies that $r \sim s$.

If $|\gamma| = 4$, then let $u = w$.
If $|\gamma| > 4$, let $u \in \gamma$ such that $l \sim u$, $d(u,w) = d(l,w) - 1$.  Note that $s,u \in X_{l} \cap B_{2k-3+|h|}(h(y))$. Then the (V) condition of the SD'$(h(y))$-property implies that there exists a vertex $p \in X_{l} \cap B_{2k-3+|h|}(h(y))$ such that $p \sim s,u$.

Because $r,s \in B_{2k-3+|h|}(h(y)), r \sim s$, the (E) condition of the SD'$(h(y))$-property implies that there exists a vertex $c \in B_{2k-4+|h|}(h(y))$ such that $c \sim r,s$.

Because $s,p \in B_{2k-3+|h|}(h(y)), s \sim p$, the (E) condition of the SD'$(h(y))$-property implies that there exists a vertex $d \in B_{2k-4+|h|}(h(y))$ such that $d \sim s,p$.

Note that $m, l \in B_{2k-2+|h|}(h(y))$, $r,s,p \in B_{2k-3+|h|}(h(y))$, $c, d \in X_{s} \cap $ $B_{2k-4+|h|}(h(y))$, $r,s \in X_{m}, s,p \in X_{l}$. Then, because $m \sim l$, Lemma \ref{2.6} implies that $c \sim d$.

Let $\delta$ be the tightening of the cycle $(x,t,r,c,d,p,l)$. Note that $|\delta| \leq 7$ and the cycle $\delta$ is full. Then, by $8$-location, there is a vertex $f$ such that $\delta \subset X_{f}$. Hence $d(x,c) = 2$. But $x \in B_{2k - 1 + |h|}(h(y))$ while $c \in B_{2k - 4 + |h|}(h(y))$. Therefore $d(x,c) = 3$. This yields a contradiction.

Case B.$2$. Assume $|\gamma| = 2k+1, k \in \mathbf{N}^{\star}$. 

Assume first $d(v,y) = k+1$. Then $d(y,w) = k$. Note that $d(y,h(y)) = 2k + 1 + |h|$ by the left of the cycle $\gamma \star \beta \star h(\gamma) \star \alpha$ and $d(y,h(y)) = 2k + |h|$ by the right of the cycle $\gamma \star \beta \star h(\gamma) \star \alpha$. So the geodesic from $y$ to $h(y)$ passes by the right of the cycle $\gamma \star \beta \star h(\gamma) \star \alpha$. But the point $y$ is chosen such that the geodesic from $y$ to $h(y)$ passes by the left of the cycle $\gamma \star \beta \star h(\gamma) \star \alpha$. The situation $d(v,y) = k+1$ is therefore not possible. 
So the only possible case is when $d(v,y) = k$. Therefore $d(y,w) = k+1$, $d(x,w) = k$. 

Note that $d(x,h(x)) = d(y,h(x)) = 2k+|h|$. Hence $y,x \in B_{2k+|h|}(h(x))$. Then, since $x \sim y$, the (E) condition of the SD'(h(x))-property implies that there exists a vertex $t \in B_{2k-1+|h|}(h(x))$ such that $t \sim y,x$.

If $|\gamma| = 3,$ let $z \in \alpha, z \sim y$.
If $|\gamma| > 3,$ let $z \in \gamma$ such that $z \sim y$, $d(z,v) = d(y,v) - 1$. Note that $z,t \in X_{y} \cap B_{2k-1+|h|}(h(x))$. Then the (V) condition of the SD'$(h(x))$-property implies that there exists a vertex $u \in X_{y} \cap B_{2k-1+|h|}(h(x))$ such that $u \sim z,t$.

Because $z,u \in B_{2k-1+|h|}(h(x)), z \sim u$, the (E) condition of the SD'$(h(x))$-property implies that there exists a vertex $p \in B_{2k-2+|h|}(h(x))$ such that $p \sim z,u$.

Because $u,t \in B_{2k-1+|h|}(h(x)), u \sim t$, the (E) condition of the SD'$(h(x))$-property implies that there exists a vertex $q \in B_{2k-2+|h|}(h(x))$ such that $q \sim u,t$.

Note that $y \in B_{2k+|h|}(h(x))$, $z,u,t \in X_{y} \cap B_{2k-1+|h|}(h(x))$ and $p,q \in X_{u} \cap B_{2k-2+|h|}(h(x))$, $p \sim z, q \sim t$. Then Lemma \ref{2.4} implies that $p \sim q$.

If $|\gamma| = 3$, let $l \in \alpha, l \sim z$. If $|\gamma| = 5$, let $l = v$. 
If $|\gamma| > 5$, let $l \in \gamma$ such that $l \sim z$, $d(l,v) = d(z,v) - 1$. Note that $l,p \in X_{z} \cap B_{2k-2+|h|}(h(x))$. Then the (V) condition of the SD'$(h(x))$-property implies that there exists a vertex $r \in X_{z} \cap B_{2k-2+|h|}(h(x))$ such that $r \sim l,p$.

Because $r,p \in B_{2k-2+|h|}(h(x)), r \sim p$, the (E) condition of the SD'$(h(x))$-property implies that there exists a vertex $n \in B_{2k-3+|h|}(h(x))$ such that $n \sim r,p$.

Because $p,q \in B_{2k-2+|h|}(h(x)), p \sim q$, the (E) condition of the SD'$(h(x))$-property implies that there exists a vertex $c \in B_{2k-3+|h|}(h(x))$ such that $c \sim p,q$.

Note that $z, u \in B_{2k-1+|h|}(h(x))$, $r,p,q \in B_{2k-2+|h|}(h(x))$, $r,p \in X_{z}, p,q \in X_{u}$, $n,c \in X_{p} \cap B_{2k-3+|h|}(h(x))$. Then, because $z \sim u$, Lemma \ref{2.6} implies that $n \sim c$.

Let $\delta$ be the tightening of the cycle $(y,z,r,n,c,q,t)$. Note that $|\delta| \leq 7$ and the cycle $\delta$ is full. Then, by $8$-location, there is a vertex $f$ such that $\delta \subset X_{f}$. Hence $d(y,n) = 2$. But $y \in B_{2k + |h|}(h(x))$ while $n \in B_{2k - 3 + |h|}(h(x))$. Therefore $d(y,n) = 3$. This yields a contradiction.

In conclusion we have $d(v',h(v')) \neq |h|+1$. This completes case $B.$
 
Case $C.$  There exists a vertex $b \in \alpha$ such that $h(v) \sim b \sim h(v')$. Suppose $d(v',h(v')) = |h|+1$.

Case $C.1$. Assume $|\gamma| = 2k, k \in \mathbf{N}^{\star}$. Assume w.l.o.g. $d(y,v) = k$. Then, due to the choice of the vertices $y,x \in \gamma$, we have $d(x,w) = k-1$. Recall $y$ is the last vertex of $\gamma$ to be joined with $h(y)$ by the left of the cycle $\gamma \star \beta \star h(\gamma) \star \alpha$; $x$ is the first vertex of $\gamma$ to be joined with $h(x)$ by the right of the cycle $\gamma \star \beta \star h(\gamma) \star \alpha$. 

Note that $d(h(y),y) = d(h(x),y) = 2k-1+|h|$. Because $h(y),h(x) \in B_{2k-1+|h|}(y)$ and $h(y) \sim h(x)$, the (E) condition of the SD'$(y)$-property implies that there exists a vertex $t \in B_{2k-2+|h|}(y)$ such that $t \sim h(y),h(x)$.

If $|\gamma| = 3$, let $l = h(w)$.
If $|\gamma| > 3$, let $l \in h(\gamma)$ such that $l \sim h(x)$, $d(l,h(w)) = d(h(x),h(w)) - 1$. Note that $t,l \in X_{h(x)} \cap B_{2k-2+|h|}(y)$. Then the (V) condition of the SD'$(y)$-property implies that there exists a vertex $u \in X_{h(x)} \cap B_{2k-2+|h|}(y)$ such that $u \sim t,l$.

Note that $t,u \in B_{2k-2+|h|}(y), t \sim u$, the (E) condition of the SD'$(y)$-property implies that there exists a vertex $p \in B_{2k-3+|h|}(y)$ such that $p \sim t,u$.

Note that $u,l \in B_{2k-2+|h|}(y), u \sim l$, the (E) condition of the SD'$(y)$-property implies that there exists a vertex $q \in B_{2k-3+|h|}(y)$ such that $q \sim u,l$.

Note that $h(x) \in B_{2k-1+|h|}(y)$, $t,u,l \in X_{h(x)} \cap B_{2k-2+|h|}(y)$ and $p,q \in X_{u} \cap B_{2k-3+|h|}(y)$, $p \sim t, q \sim l$. Then Lemma \ref{2.5} implies that $p \sim q$.

If $|\gamma| = 4$, then $z \in \beta$, $z \sim l$. If $|\gamma| = 6$, then  $z = w$. 
If $|\gamma| > 6$, let $z \in h(\gamma)$ such that $z \sim l$, $d(z,h(w)) = d(l,h(w)) - 1$. Note that $q,z \in X_{l} \cap B_{2k-3+|h|}(y)$. Then the (V) condition of the SD'$(y)$-property implies that there exists a vertex $n \in X_{l} \cap B_{2k-3+|h|}(y)$ such that $n \sim q,z$.

Because $p,q \in B_{2k-3+|h|}(y), p \sim q$, the (E) condition of the SD'$(y)$-property implies that there exists a vertex $r \in B_{2k-4+|h|}(y)$ such that $r \sim p,q$.

Because $q,n \in B_{2k-3+|h|}(y), q \sim n$, the (E) condition of the SD'$(y)$-property implies that there exists a vertex $c \in B_{2k-4+|h|}(y)$ such that $c \sim q,n$.

Note that $u, l \in B_{2k-2+|h|}(y)$, $p,q,n \in B_{2k-3+|h|}(y)$, $p,q \in X_{u}, q,n \in X_{l}$, $r,c \in X_{q} \cap B_{2k-4+|h|}(y), p \sim r, n \sim c$. Then, because $u \sim l$, Lemma \ref{2.6} implies that $r \sim c$.

Let $\delta$ be the tightening of the cycle $(h(x),t,p,r,c,n,l)$. Note that $|\delta| \leq 7$ and the cycle $\delta$ is full. Then, by $8$-location, there is a vertex $f$ such that $\delta \subset X_{f}$. Hence $d(h(x),r) = 2$. But $h(x) \in B_{2k -1 + |h|}(y)$ while $r \in B_{2k - 4 + |h|}(y)$. Therefore $d(h(x),r) = 3$. This yields a contradiction.

Case $C.2$. Assume $|\gamma| = 2k + 1, k \in \mathbf{N}^{\star}$. %Then, due to the choice of the vertices $x$ and $y$, we have %$d(y,v) = d(x,w) = k$. This is the only possible situation due %to the choice of the vertices $x,y \in \gamma$ ($y$ is the last %vertex of $\gamma$ to be joined with $h(y)$ by the left of the %cycle $\gamma \star \beta \star h(\gamma) \star \alpha$; $x$ is %the first vertex of $\gamma$ to be joined with $h(x)$ by the %right of the cycle $\gamma \star \beta \star h(\gamma) \star %\alpha$).

Assume first $d(v,y) = k+1$. Then $d(y,w) = k$. Note that $d(y,h(y)) = 2k + 1 + |h|$ by the left of the cycle $\gamma \star \beta \star h(\gamma) \star \alpha$ and $d(y,h(y)) = 2k + |h|$ by the right of the cycle $\gamma \star \beta \star h(\gamma) \star \alpha$. So the geodesic from $y$ to $h(y)$ passes by the right of the cycle $\gamma \star \beta \star h(\gamma) \star \alpha$. But the point $y$ is chosen such that the geodesic from $y$ to $h(y)$ passes by the left of the cycle $\gamma \star \beta \star h(\gamma) \star \alpha$. The situation $d(v,y) = k+1$ is therefore not possible. 
So the only possible case is when $d(v,y) = k$. Therefore $d(y,w) = k+1$, $d(x,w) = k$.

Note that $d(h(x),x) = d(h(y),x) = 2k + |h|$. Because $h(y),h(x) \in B_{2k+|h|}(x)$ and $h(y) \sim h(x)$, the (E) condition of the SD'$(x)$-property implies that there exists a vertex $t \in B_{2k-1+|h|}(x)$ such that $t \sim h(y),h(x)$.

If $|\gamma| = 3$, let $l = h(w)$.
If $|\gamma| > 3$, let $l \in h(\gamma)$ such that $l \sim h(x)$, $d(l,h(w)) = d(h(x),h(w)) - 1$. Note that $t,l \in X_{h(x)} \cap B_{2k-1+|h|}(x)$. Then the (V) condition of the SD'$(x)$-property implies that there exists a vertex $u \in X_{h(x)} \cap B_{2k-1+|h|}(x)$ such that $u \sim t,l$.

Because $t,u \in B_{2k-1+|h|}(x), t \sim u$, the (E) condition of the SD'$(x)$-property implies that there exists a vertex $p \in B_{2k-2+|h|}(x)$ such that $p \sim t,u$.

Because $u,l \in B_{2k-1+|h|}(x), u \sim l$, the (E) condition of the SD'$(x)$-property implies that there exists a vertex $q \in B_{2k-2+|h|}(x)$ such that $q \sim u,l$.

Note that $h(x) \in B_{2k+|h|}(x)$, $t,u,l \in X_{h(x)} \cap B_{2k-1+|h|}(x)$, $p,q \in X_{u} \cap B_{2k-2+|h|}(x)$, $p \sim t, q \sim l$. Then Lemma \ref{2.5} implies that $p \sim q$.

If $|\gamma| = 3$, let $z \in \beta$. 
If $|\gamma| = 5$, let $z = h(w)$.
If $|\gamma| > 5$, let $z \in h(\gamma)$ such that $z \sim l$, $d(z,h(w)) = d(l,h(w)) - 1$. Note that $q,z \in X_{l} \cap B_{2k-2+|h|}(x)$. Then the (V) condition of the SD'$(x)$-property implies that there exists a vertex $s \in X_{l} \cap B_{2k-2+|h|}(x)$ such that $s \sim q,z$.

Because $p,q \in B_{2k-2+|h|}(x), p \sim q$, the (E) condition of the SD'$(x)$-property implies that there exists a vertex $c \in B_{2k-3+|h|}(x)$ such that $c \sim p,q$.

Because $q,s \in B_{2k-2+|h|}(x), q \sim s$, the (E) condition of the SD'$(x)$-property implies that there exists a vertex $d \in B_{2k-3+|h|}(x)$ such that $d \sim q,s$.

Note that $u, l \in B_{2k-1+|h|}(x)$, $p,q,s \in B_{2k-2+|h|}(x)$, $p,q \in X_{u}, q,s \in X_{l}$, $c,d \in X_{q} \cap B_{2k-3+|h|}(y)$. Then, because $u \sim l$, Lemma \ref{2.6} implies that $c \sim d$.

Let $\delta$ be the tightening of the cycle $(h(x),t,p,c,d,s,l)$. Note that $|\delta| \leq 7$ and the cycle $\delta$ is full. Then, by $8$-location, there is a vertex $f$ such that $\delta \subset X_{f}$. Hence $d(h(x),c) = 2$. But $h(x) \in B_{2k + |h|}(x)$ while $c \in B_{2k - 3 + |h|}(x)$. Therefore $d(h(x),c) = 3$. This yields a contradiction.

In conclusion we have $d(v',h(v')) \neq |h|+1$. This completes case $C.$

Case $D.$ There exist vertices $a,b \in \alpha$ such that $v' \sim a \sim v$, $h(v') \sim b \sim h(v)$. Then $d(v',h(v')) = |h|$ which yields a contradiction.

\end{proof}

\begin{lemma}\label{3.3}
Let $h$ be a (simplicial) isometry with no fixed points of an $8$-located complex $X$ with the SD'-property. Let $Y = \rm{Min}_{X}(h)$. Then $Y = \rm{Min}_{Y}(h)$.
\end{lemma}

\begin{proof}
Let $x \in X$ such that $d_{X}(x,h(x)) = |h|$. Then $x \in Y$. Let $y = h(x) \in Y$ such that $d_{Y}(y,h(y)) = |h|$. So $y \in \rm{Min}_{Y}(h)$. Since $d_{X}(x,h(x)) = d_{Y}(y,h(y))$, we have $Y = \rm{Min}_{Y}(h)$.

\end{proof}

The construction of a low-dimensional classifying space for the family of virtually cyclic subgroups of a group acting properly on an $8$-located complex with the $SD'$-property relies on the following result.

\begin{theorem}\label{3.4}
Let $h$ be a (simplicial) isometry having no fixed points with $|h| > 3$, of an $8$-located complex $X$ with the SD'-property. Then
the set $\rm{Min}_{X}(h)$ is Gromov hyperbolic. In particular, $\rm{Min}_{X}(h)$ is systolic.
\end{theorem}

\begin{proof}
Theorem \ref{2.4} implies that $X$ is Gromov hyperbolic.
Let $Y = \rm{Min}_{X}(h)$.
Lemma \ref{3.3} implies that $Y = \rm{Min}_{Y}(h)$. The proof is by contradiction.
Suppose there exists a $k$-wheel $\gamma = (z;x_{1}, ..., x_{k}) \subset Y$, $5 \leq k \leq 6$. According to Lemma \ref{3.2}, the $1$-skeleton of $Y$ is isometrically embedded into $X$. Then the $k$-wheel $\gamma$ also belongs to $X$. Due to the Gromov hyperbolicity of $X$, this yields a contradiction.
So there does not exist any $k$-wheel in $Y$, $5 \leq k \leq 6$. This implies that $Y$ is  Gromov hyperbolic. In particular, $Y$ is systolic.

\end{proof}

%\begin{theorem}\label{3.5}
%Let $h$ be a (simplicial) isometry of a weakly systolic complex $X$ %having no fix-points such that
%the set $\rm{Min}(h) \subset X$ is isometrically embedded and it %satisfies the extended $5$-wheel condition (i.e. $(C_{5},W_{5})$, %$\hat{W}_{5}$-condition). Then $\rm{Min}(h) \subset X$ is weakly %systolic.
%\end{theorem}

%\begin{proof}

%There are no $5$-wheels by hypothesis. There are no full $4$-cycles %by weak systolicity. 

%\end{proof}

The following results on $8$-located complexes with the SD'-property are immediate consequences of the fact that the minimal displacement set of a nonelliptic isometry acting on such complex is a systolic subcomplex and it embeds isometrically into the complex. Their systolic analogues, also given below, imply these similarities. We shall refer to these results when constructing a low-dimensional classifying space for the family of virtually cyclic subgroups of a group acting properly on an $8$-located complex with the SD'-property.

\begin{theorem}\label{3.5}
Let $h$ be a nonelliptic simplicial isometry of a uniformly
locally finite systolic complex $X$. Then there is an $h^{n}$-invariant geodesic for some $n \geq 1$.
\end{theorem}

For the proof see \cite{E2}, Theorem $3.5$, page $46$.

\begin{theorem}\label{3.6}
Let $h$ be a nonelliptic simplicial isometry of a uniformly
locally finite $8$-located complex $X$ with the SD'-property. Assume $|h| > 3$. Then in $X$ there is an $h^{n}$-invariant geodesic for some $n \geq 1$.
\end{theorem}

\begin{proof}

Let $Y = \rm{Min}_{X} (h)$. Since $|h| > 3$, Theorem \ref{3.4} implies that $Y$ is systolic. Then, by Theorem \ref{3.5}, there is in $Y$ an $h^{n}$-invariant geodesic $\gamma$ for some $n \geq 1$. Since, by Theorem \ref{3.2}, $Y^{(1)}$ is isometrically embedded into $X$, the $h^{n}$-invariant geodesic $\gamma$ also belongs to $X$. This completes the proof.

\end{proof}

\begin{theorem}\label{3.7}
Let $h$ be a simplicial isometry of a uniformly locally finite
systolic complex $X$. Then either there is an $h$-invariant simplex (elliptic case)
or there is an $h$-invariant thick geodesic (hyperbolic case).
\end{theorem}

For the proof see \cite{E2}, Theorem $3.8$, page $49$.

\begin{theorem}\label{3.8}
Let $h$ be a simplicial isometry of a uniformly locally finite $8$-located complex $X$ with the SD'-property. Assume $|h| > 3$. Then either there is an $h$-invariant simplex (elliptic case)
or there is an $h$-invariant thick geodesic (hyperbolic case).
\end{theorem}

\begin{proof}

Let $Y = \rm{Min}_{X} (h)$. Since $|h| > 3$, Theorem \ref{3.4} implies that $Y$ is systolic. Then, by Theorem \ref{3.7}, in $Y$ either there is an $h$-invariant simplex (elliptic case)
or there is an $h$-invariant thick geodesic (hyperbolic case). Since, by Theorem \ref{3.2}, $Y^{(1)}$ is isometrically embedded into $X$, this $h$-invariant simplex (elliptic case), respectively this $h$-invariant thick geodesic (hyperbolic case) also belongs to $X$. 

\end{proof}

\begin{theorem}\label{3.9}
Let $h$ be a nonelliptic simplicial isometry of a uniformly locally finite systolic complex $X$.
If there exists an $h^{n}$-invariant geodesic in $X$, then for any vertex $x \in \rm{Min}_{X}(h^{n}) \subset X$, there exists an $h^{n}$-invariant geodesic passing through $x$.
\end{theorem}

For the proof see \cite{E2}, Remark page $48$.

\begin{theorem}\label{3.10}
Let $h$ be a nonelliptic simplicial isometry of a uniformly locally finite $8$-located complex $X$ with the SD'-property. Assume $|h| > 3$. 
If there exists an $h^{n}$-invariant geodesic in $X$, then for any vertex $x \in \rm{Min}_{X}(h^{n}) \subset X$, there exists an $h^{n}$-invariant geodesic passing through $x$.
\end{theorem}

\begin{proof}
Let $Y = \rm{Min}_{X}(h)$. Theorem \ref{3.4} implies that $Y$ is systolic.
According to Theorem \ref{3.5}, in $Y$ (and then, by Theorem $3.6$, also in $X$) there exists an $h^{n}$-invariant geodesic for some $n \geq 1$.  Hence, by Theorem \ref{3.9}, for any vertex $x \in \rm{Min}_{X}(h^{n}) \subset Y$, there exists an $h^{n}$-invariant geodesic passing through $x$. Since, by Theorem \ref{3.2}, $Y^{(1)}$ is isometrically embedded into $X$, this implies that for any vertex $x \in \rm{Min}_{X}(h^{n}) \subset Y \subset X$, there exists an  $h^{n}$-invariant geodesic passing through $x$.
\end{proof}

 \section{Classifying spaces with virtually cyclic stabilisers for $8$-located groups}
 
In this section we construct a low-dimensional classifying space for the family of virtually cyclic subgroups of a group acting properly on an $8$-located complex with the $SD'$-property.
The proof relies on the fact that the minimal displacement set of such complex is a systolic subcomplex that embeds isometrically into the complex. We start by giving the systolic analogue of one of the main results the construction will be based on. 

%\begin{theorem}\label{3.11}
%Let $h$ be a hyperbolic isometry of a uniformly locally finite %systolic
%complex $X$, such that $|h| > 3$, and the subcomplex %$\rm{Min}_{X}(h)$
%is the union of axes. Then there is a quasi-isometry
%$c : (Y(h) \times \mathbf{Z},d_{h}) \rightarrow %(\rm{Min}_{X}(h),d_{X})$
%where the metric $d_{h}$ is defined as
%$d_{h}((\gamma_{1},t_{1}), (\gamma_{2},t_{2})) = %d_{Y(h)}(\gamma_{1},\gamma_{2}) + |t_{1}-t_{2}|$,
%and $d_{X}$ is the metric induced from $X$.
%\end{theorem}

%For the proof see \cite{OP}, Theorem $3.3$, page $12$.

%\begin{theorem}\label{3.12}
%Let $h$ be a hyperbolic isometry of a uniformly locally finite %$8$-located complex with the $SD'$-property, such that $|h| > %3$, and the subcomplex $\rm{Min}_{X}(h)$
%is the union of axes. Then there is a quasi-isometry
%$c : (Y(h) \times \mathbf{Z},d_{h}) \rightarrow %(\rm{Min}_{X}(h),d_{X})$
%where the metric $d_{h}$ is defined as
%$d_{h}((\gamma_{1},t_{1}), (\gamma_{2},t_{2})) = %d_{Y(h)}(\gamma_{1},\gamma_{2}) + |t_{1}-t_{2}|$,
%and $d_{X}$ is the metric induced from $X$.
%\end{theorem}

%\begin{proof}

%\end{proof}

\begin{theorem}\label{3.13}
Let $X$ be a systolic locally finite simplicial complex.
For a hyperbolic isometry $h$ whose minimal displacement set is
a union of axes (that is, for $h$ satisfying $(2.1)$) and $|h| > 3$, the graph of axes $(Y(h),d_{Y(h)})$ is quasi-isometric to a simplicial tree.
\end{theorem}

For the proof see \cite{OP}, Corollary $4.6$, page $21$.

\begin{theorem}\label{3.14}
Let $X$ be a locally finite $8$-located complexes with the $SD'$-property.
For a hyperbolic isometry $h$ whose minimal displacement set is
a union of axes (that is, for $h$ satisfying $(2.1)$) and $|h| > 3$, the graph of axes $(Y(h),d_{Y(h)})$ is quasi-isometric to a simplicial tree.
\end{theorem}

\begin{proof}
Let $Y = \rm{Min}_{X}(h)$. Lemma \ref{3.3} implies that $Y = \rm{Min}_{Y}(h)$.
If there do not exist $h$-invariant geodesics in $X$, take an $h^{n}$-invariant geodesic in $X$, $n > 1$ (see Lemma $4.9$). Assume there exist $h$-invariant geodesics in $X$. These geodesics are also in $Y$ because, according to $(2.1)$, \begin{center}$Y = \rm{span}\{\bigcup \gamma | \gamma$ is an $h$-invariant geodesic \}.\end{center} Theorems \ref{3.2} and \ref{3.3} imply that $Y$ is systolic and $Y^{(1)}$ embeds isometrically into $X$. Then, by Theorem \ref{3.13}, the result follows.

\end{proof}

For the rest of the section,
let $G$ be a group acting properly discontinuously on a uniformly locally
finite $8$-located complex $X$ with the $SD'$-property of dimension $d$.

\begin{theorem}\label{4.1}
 The systolic complex $X$ is a model for $\underline{E}G$.
 \end{theorem}
 
 For the proof see \cite{ChOs}, Theorem E.

In order to construct models for the commensurators $N_{G}[H]$, first we show that the group $G$ satisfies condition (C). Using this, in every finitely generated subgroup $K \subseteq N_{G}[H]$ that contains $H$ we find a suitable normal cyclic subgroup. As shown in \cite{OP} for systolic complexes, the quotient
group acts properly on a quasi-tree. 

\begin{lemma}\label{4.2}
The group $G$ satisfies condition (C).
\end{lemma}

\begin{proof}
 The proof is similar to the one given in \cite{OP} (Lemma $5.2$). Take arbitrary $g,h \in G$ such that $|h| = \infty$, and assume there are $k,l \in \mathbf{Z}$ such that $g^{-1} h^{k} g = h^{l}$. We have to show that $|k| = |l|$. Since the action of $G$ on
$X$ is proper, the element $h$ acts as a hyperbolic isometry. By Theorem \ref{3.6}, there is in $X$ an $h^{n}$-invariant geodesic for some $n \geq 1$. We get the claim by
considering the following sequence of equalities for the translation length:
\begin{center}
    $|k| \cdot |h^{n}| =  |h^{k \cdot n}| =  |g^{-1} \cdot h^{n \cdot k} \cdot g| =  |h^{\pm l \cdot n}| = |l| \cdot |h^{n}|$.
\end{center}
The first and the last of the equalities follow from the fact that the translation length of an element can be measured on an invariant geodesic, the second one is an easy calculation and the third one is straightforward.
\end{proof}

\begin{lemma}\label{4.3}
Let $K$ be a finitely generated subgroup of $G$, and $h \in K$ a hyperbolic
isometry satisfying $(2.1)$, such that $\langle h \rangle$ is normal in $K$. Then the proper action of
$G$ on $X$ induces a proper action of $K / \langle h \rangle$ on the graph of axes $Y(h)$.
\end{lemma}

\begin{proof}
The proof is similar to the one given in \cite{OP}, Lemma $5.3$ in the systolic case.
\end{proof}

\begin{lemma}\label{4.4}
Let $h$ be a hyperbolic isometry of an $8$-located complex with the $SD'$-property $X$. Assume that $|h| > 3$. Then if $h$ satisfies $(2.1)$ then so does $h^{n}$
for any $n \in \mathbf{Z} \setminus \{ 0 \}$.
\end{lemma}

\begin{proof}
The result follows by Lemma \ref{3.10} and the fact that an $h$-invariant geodesic is $h^{n}$-invariant.
\end{proof}

\begin{lemma}\label{4.5} 
Let $K$ be a finitely generated subgroup of $N_{G}[H]$ that contains $H$.
Then there is a short exact sequence
\begin{center}
    $0 \rightarrow \langle h \rangle \rightarrow K \rightarrow K / \langle h \rangle \rightarrow 0,$
\end{center}
such that $h \in H$ is of infinite order and the group $K / \langle h \rangle$ is virtually free.

\end{lemma}

\begin{proof}
The proof is similar to the one given in \cite{OP}, Lemma $5.4$.
Choose an element of infinite order $\tilde{h} \in H$ satisfying the following two
conditions:
\begin{enumerate}
\item the set $\rm{Min}_{X}(\tilde{h})$ is the union of axes (see $(2.1)$);
\item the translation length $|\tilde{h}| > 3$.
\end{enumerate}

Both conditions above can be ensured by rising $\tilde{h}$
to a sufficiently large power. Indeed,
by Lemma \ref{3.6}, there exists $n \geq 1$ such that $\tilde{h}^{n}$
satisfies the first condition above. If
$|\tilde{h}^{n}| \leq 3$ then replace it with $\tilde{h}^{4n}$. The element $\tilde{h}^{4n}$ satisfies both conditions. If an element satisfies the conditions above then, by Lemma \ref{4.4}, so does any of its powers. Since $G$ satisfies condition (C),
by Lemma \ref{2.1}, there exists an integer $k \geq 1$ such that $\langle  \tilde{h}^{k} \rangle$ is normal in $K$.

Put $h = \tilde{h}^{k}$. By Lemma \ref{4.3}, the group $K / \langle h \rangle$ acts properly by isometries on
the graph of axes $(Y(h),d_{Y(h)})$, which, by Theorem \ref{3.14}, is a quasi-tree. In conclusion the group $K / \langle h \rangle$ is virtually free.

\end{proof}

The proofs of the next results are similar to the one given for systolic complexes in \cite{OP} (see Lemma $5.5$, Theorem C).

\begin{lemma}
For every $[H] \in \mathcal{VCY} \setminus \mathcal{FIN}$ there exist:
\begin{enumerate}
\item a $2$-dimensional model for $E_{\mathcal{G}[H]} N_{G}[H]$;
\item a $3$-dimensional model for $\underline{E} N_{G}[H]$.
\end{enumerate}
\end{lemma}

\begin{theorem}
There exists a model for
$\underline{\underline{E}}G$ of dimension
\begin{center}
    $\dim \underline{\underline{E}}G =$ 
    $\begin{cases} 
    $d+1$, & \text{if $d \leq 3$,}\\ 
    $d$, & \text{if $d \geq 4$.}
    \end{cases}$
\end{center}
\end{theorem}


\begin{thebibliography}{10}


%\bibitem{AB}
%K. A. Adiprasito, B. Benedetti, \emph{Collapsibility of CAT(0) %spaces}, arXiv:1107.5789v8, 2019.

\bibitem{BCCGO}
B. Bre{\v{s}}ar, J. Chalopin, V. Chepoi, T. Gologranc and D. Osajda,
\emph{Bucolic complexes}
Adv. Math., {\bf 243}, 2013,
pp. 127-167.



\bibitem{BH}
M. Bridson, A. Haefliger,
\emph{Metric spaces of non-positive curvature}
Springer, New York, 1999.


\bibitem{Ch}
V. Chepoi,
\emph{Graphs of some CAT(0) complexes},
Adv. in Appl. Math., {\bf 24}, 2000, no. 2,
pp. 125-179.


\bibitem{ChaCHO}
J. Chalopin, V. Chepoi, H. Hirai and D. Osajda,
\emph{Weakly modular graphs and nonpositive curvature},
preprint, arXiv:1409.3892, 2014.

\bibitem{ChOs}
V. Chepoi and D. Osajda,
\emph{Dismantlability of weakly systolic complexes and applications},
Trans. Amer. Math. Soc., {\bf 367}, 2015, no. 2,
pp. 1247-1272.


%\bibitem{crowley_2008}

%K. Crowley, \emph{ Discrete Morse theory and the geometry of
%nonpositively curved simplicial complexes}, Geometriae Dedicata, %{\bf 133}, 2008, no. 1, 35-50.

\bibitem{Die10}
R. Diestel,
\emph{Graph theory},
Graduate Texts in Mathematics, {\bf 173}, Springer, Heidelberg, 2010.


\bibitem{E2}
T. Elsener,
\emph{Isometries of systolic spaces},
Fundamenta Mathematicae, {\bf 204}, 2009, 39-55.


\bibitem{E1}
T. Elsener,
\emph{Flats and the flat torus theorem for systolic spaces},
Geometry and Topology, {\bf 13}, 2009, 661-698.



\bibitem{Gom}
R. G{\'o}mez-Ortells,
\emph{Compactly supported cohomology of systolic 3-pseudomanifolds},
Colloq. Math., {\bf 135}, 2014, no. 1,
pp. 103-112.

\bibitem{Gro}
M. Gromov,
\emph{Hyperbolic groups},
Math. Sci. Res. Inst. Publ., {\bf 8}, Springer, New York, 1987,
pp. 75-263.

\bibitem{Hag}
F. Haglund,
\emph{Complexes simpliciaux hyperboliques de grande dimension},
preprint, 2003, \url{http://www.math.u-psud.fr/~haglund/cpl_hyp_gde_dim.pdf}

\bibitem{JS0}
T. Januszkiewicz and J. {\'S}wi{\c{a}}tkowski,
\emph{Hyperbolic Coxeter groups of large dimension},
Comment. Math. Helv., {\bf 78}, $2003$, no. 3,
pp. 555-583.

\bibitem{JS1}
T. Januszkiewicz and J. {\'S}wi{\c{a}}tkowski,
\emph{Simplicial nonpositive curvature},
Publ. Math. Inst. Hautes \'Etudes Sci., no. $104$, $2006$,
pp. 1-85.

\bibitem{JS2}
T. Januszkiewicz and J. {\'S}wi{\c{a}}tkowski,
\emph{Filling invariants of systolic complexes and groups},
Geom. Topol., {\bf 11}, $200$,
pp. 727-758.





\bibitem{lazar_2013}
I.-C. Laz\u{a}r,
\emph{Systolic simplicial complexes are collapsible},
Bull. Math. Soc. Sci. Math. Roumanie,, $56(104)$, {\bf  2}, $2013$,  pp. $229-236$.


\bibitem{L-8loc}
I.-C. Laz\u{a}r,
\emph{A combinatorial negative curvature condition implying Gromov hyperbolicity},
preprint, arXiv:$1501.05487v4$, $2015$.


\bibitem{L-8loc2}
I.-C. Laz\u{a}r,
\emph{Minimal disc diagrams of $5/9$-simplicial complexes},
Michigan Math. J., $69$, $2020$,  pp. $793-829$.

\bibitem{L}
W. L\"uck, \emph{Survey on classifying spaces for families of subgroups}. L. Bartholdi,
T. Ceccherini-Silberstein, T. Smirnova-Nagnibeda and A. Zuk (eds.), \emph{Infinite groups:
geometric, combinatorial and dynamical aspects}. (Gaeta, $2003$.) Progress in Mathematics, $248$. Birkhäuser, Basel, $2005$, 269-322.

\bibitem{LW}
 W. L\"uck, M. Weiermann, \emph{On the classifying space of the family of virtually cyclic subgroups}, Pure Appl. Math. Q., {\bf 8}, 2012, no. 2, pp. 497-555.

\bibitem{O-ci}
D. Osajda,
\emph{Connectedness at infinity of systolic complexes and groups},
Groups Geom. Dyn., {\bf 1}, 2007, no. 2,
pp. 183-203.

\bibitem{O-ib}
D. Osajda,
\emph{Ideal boundary of 7-systolic complexes and groups},
Algebr. Geom. Topol., {\bf 8}, 2008, no. 1,
pp. 81-99.

\bibitem{O-chcg}
D. Osajda,
\emph{A construction of hyperbolic Coxeter groups},
Comment. Math. Helv., {\bf 88}, 2013, no. 2,
pp. 353-367.

\bibitem{O-sdn}
D. Osajda,
\emph{A combinatorial non-positive curvature I: weak systolicity},
preprint, arXiv: 1305.4661, 2013.

\bibitem{O-8loc}
D. Osajda,
\emph{Combinatorial negative curvature and triangulations of three-manifolds},
Indiana Univ. Math. J., {\bf 64}, 2015, no. 3,
pp. 943-956.

\bibitem{O-ns}
D. Osajda,
\emph{Normal subgroups of SimpHAtic groups},
submitted, arXiv:1501.00951, 2015.


\bibitem{OP}
D. Osajda and T. Prytu{\l}a,
\emph{Classifying spaces for families of subgroups for systolic groups},
Groups Geom. Dyn. 12, 2018, no. 3, 1005-1060.


\bibitem{OS}
D. Osajda and J. {\'S}wi{\c{a}}tkowski,
\emph{On asymptotically hereditarily aspherical groups},
Proc. London Math. Soc. 111, 2015, no. 1, 93-126.


%\bibitem{Papa}
%P. Papasoglu,
%\emph{Strongly geodesically automatic groups are hyperbolic},
%Invent. Math. {\bf 121}, 1995, no. 2,
%pp. 323-334.


\bibitem{Pr}
T. Prytu{\l}a,
\emph{Infinite systolic groups are not torsion},
Colloquium Mathematicum {\bf 153}, 2018, no. 2,
pp. 169-194.

\bibitem{Prz}
P. Przytycki, \emph{The fixed point theorem for simplicial nonpositive curvature},
Math. Proceedings of the Cambridge Philosophical Society, {\bf 144}, 2008, no. 3, pp. 683-695.
\end{thebibliography}
\end{document}